\documentclass[runningheads,a4paper]{article}
\usepackage{fullpage}
\usepackage{latexsym,graphicx}
\usepackage{amsthm,amsmath,amssymb,enumerate}
\usepackage{color}
\usepackage{booktabs}
\usepackage{bm}
\usepackage{algorithm}
\usepackage{algpseudocode}
\usepackage{xspace}
\usepackage{authblk}

\newtheorem{theorem}{Theorem}[section]
\newtheorem{definition}[theorem]{Definition}
\newtheorem{lemma}[theorem]{Lemma}
\newtheorem{remark}[theorem]{Remark}
\newtheorem{corollary}[theorem]{Corollary}

\newcounter{mynotes}
\newcommand{\remove}[1]{}

\newcommand{\R}[0]{{\ensuremath{\mathbb{R}}}}

\newcommand{\bX}[0]{\boldsymbol{X}}

\newcommand{\bZ}[0]{\boldsymbol{Z}}

\newcommand{\bx}[0]{\boldsymbol{x}}

\newcommand{\bI}{\mathbf{I}}
\newcommand{\x}{\mathbf{x}}
\newcommand{\y}{\mathbf{y}}
\newcommand{\J}{\mathbf{J}}

\newcommand{\tx}{\widetilde{\x}}
\newcommand{\bz}[0]{\boldsymbol{z}}
\newcommand{\bi}{\mathbf{i}}

\DeclareMathOperator{\proj}{proj}
\newcommand{\projb}{\proj_{bin}}

\newcommand{\altProj}{\textrm{AltProj}}
\newcommand{\perturb}[1]{\textsc{RandWalkSAT}_{#1}}

\newcommand{\lproj}{\textrm{$\ell_1$-proj}}
\newcommand{\round}{\textrm{round}}
\newcommand{\flip}{\textrm{flip}}

\newcommand{\supp}{\textrm{supp}}
\newcommand{\certsup}{\textrm{certSupp}}

\newcommand{\wt}[1]{\widetilde{#1}}

\newcommand{\wSAT}{\textsc{WalkSAT}\xspace}
\newcommand{\mbw}{\textsc{mbWalkSAT}\xspace}

\newcommand{\FPorig}{\textsc{FPorig}\xspace}
\newcommand{\WFP}{\textsc{WFP}\xspace}
\newcommand{\WFPbase}{\textsc{WFPbase}\xspace}

\newcommand{\mqed}{}

\usepackage{url}

\begin{document}

\title{Improving the Randomization Step in Feasibility Pump}


\author[1]{Santanu S. Dey\thanks{santanu.dey@isye.gatech.edu}}
\author[1]{Andres Iroume\thanks{airoume3@gatech.edu}}
\author[2]{Marco Molinaro\thanks{molinaro.marco@gmail.edu}}
\author[3]{Domenico Salvagnin\thanks{dominiqs@gmail.com}}
\affil[1]{School of Industrial and Systems Engineering, Georgia Institute of Technology, Atlanta, United States}
\affil[2]{Computer Science Department, PUC-Rio, Brazil}
\affil[3]{IBM Italy and DEI, University of Padova, Italy}	
	%


%

%
%

\maketitle

\begin{abstract}
Feasibility pump (FP) is a successful primal heuristic for mixed-integer linear programs (MILP). The algorithm consists of three main components: rounding fractional solution to a mixed-integer one, projection of infeasible solutions to the LP relaxation, and a randomization step used when the algorithm stalls. While many generalizations and improvements to the original Feasibility Pump have been proposed, they mainly focus on the rounding and projection steps.

We start a more in-depth study of the randomization step in Feasibility Pump. For that, we propose a new randomization step based on the WalkSAT algorithm for solving SAT instances. First, we provide \emph{theoretical analyses} that show the potential of this randomization step; to the best of our knowledge, this is the first time any theoretical analysis of running-time of Feasibility Pump or its variants has been conducted. Moreover, we also conduct computational experiments incorporating the proposed modification into a state-of-the-art Feasibility Pump code that reinforce the practical value of the new randomization step.

\end{abstract}



\section{Introduction}

Primal heuristics are used within mixed-integer linear programming (MILP) solvers for finding good integer feasible solutions quickly~\cite{lodiF:2011}. \emph{Feasibility pump} (FP) is a very successful primal heuristic for mixed-binary LPs that was introduced in~\cite{FischettiGL05}. At its core, Feasibility Pump is an \emph{alternating projection method}, as described below.

\begin{algorithm}[H]
\caption{Feasibility Pump (Na{\"i}ve version)}
\label{hello}
		\begin{algorithmic}[1]
		\State \textbf{Input:} mixed-binary LP (with binary variables $x$ and continuous variables $y$)

    \smallskip
    	\State Solve the linear programming relaxation, and let ${(\bar{x}, \bar{y})}$ be an optimal solution
				\While{$\bar{x}$ is not integral}
					\State (Round) Round each coordinate of $\bar{x}$ to the closest integer, call the obtained vector $\wt{x}$	\label{alg:naiveRound}
					\State (Project) Let $(\bar{x}, \bar{y})$ be the point in the LP relaxation that minimizes
					$\sum_i |{x}_i - \wt{x}_i|$ \label{alg:naiveProj}
				\EndWhile
				\State Return $(\bar{x}, \bar{y})$
    \end{algorithmic}
  \end{algorithm}
The scheme presented above may \emph{stall}, since the same infeasible integer point may be visited in Step \ref{alg:naiveRound} at different iterations. Whenever this happens, the paper~\cite{FischettiGL05} recommends a \emph{randomization step}, that after Step \ref{alg:naiveRound} flips the value of some of the binary variables as follows: Defining the \emph{fractionality} of variable $x_i$ as $|\bar{x}_i - \tilde{x}_i|$ and let $NN$ be the number of variables with positive fractionality, randomly generate a positive integer $TT$ and flip $\textup{min}\{TT, NN\}$ variables with largest fractionality.


Together with a few other tweaks, this surprisingly simple method works very well. On MIPLIB 2003 instances, FP finds feasible solutions for $96.3\%$ of the instances in reasonable time~\cite{FischettiGL05}.

Due to its success, many improvements and generalizations of FP (both for MILPs and mixed integer non-linear programs(MINLPs)) have been studied~\cite{AchterbergB07,BertaccoFL07,BonamiCLM09,FischettiS09,santis:lu:ri:2010,DAmbrosioFLL10,BolandEET12,DAmbrosioFLL12,BolandEEFST14}.
However, the focus of these improvements has been on the projection and rounding steps or generalization for MINLPs; to the best of our knowledge, they use essentially the same randomization step as proposed in the original algorithm~\cite{FischettiGL05} (and its generalization to the general integer MILP case of \cite{BertaccoFL07}).

Moreover, even though FP is so successful and so many variants have been proposed, there is very limited theoretical analysis of its properties~\cite{BolandEET12}. In particular, to the best of our knowledge there is no known bounds on expected running-time of FP.

\section{Our contributions} \label{sec:contrib}

In this paper, we start a more in-depth study of the randomization step in Feasibility Pump. For that, we propose a new randomization step $\perturb{\ell}$ and provide both \emph{theoretical analysis} as well as \emph{computational experiments} in a state-of-the-art Feasibility Pump code that show the potential of this method.

\paragraph{Theoretical justification of $\perturb{\ell}$.}

The new randomization step $\perturb{\ell}$
is inspired by the classical algorithm \emph{\wSAT}~\cite{Schoning99} for solving SAT instances (see also~\cite{Papadimitriou91,MintonJPL92}).
The key idea of $\perturb{\ell}$ is that whenever Feasibility Pump stalls, namely an infeasible mixed-binary solution is revisited, it should flip a binary variable that participates in an \emph{infeasible constraint}. More precisely, $\perturb{\ell}$ constructs a \emph{minimal (projected) infeasibility certificate} for this solution and \emph{randomly picks} a binary variable in it to be flipped (see Section \ref{sec:WalkSAT} for exact definitions).

	While the vague intuition that such randomization is trying to ``fix'' the infeasible constraint is clear, we go further and provide theoretical analyses that formally justify this and highlight more subtle advantageous properties of $\perturb{\ell}$.

	First, we analyze what happens if we simply repeatedly use \emph{only} the new proposed randomization step $\perturb{\ell}$, which gives a simple primal heuristic that we denote by \mbw. Not only we show that \mbw is guaranteed to find a solution if one exists, but its behavior is related to the \emph{(almost) decomposability} and \emph{sparsity} of the instance. To make this precise, consider a decomposable mixed-binary set with $k$ blocks:
	\begin{gather}
		P^I = P^I_1 \times \ldots \times P^I_k \textrm{, where for all $i \in [k]$ we have} \notag\\
		P^I_i = P _i \cap (\{0,1\}^{n_i} \times \R^{d_i}) \textrm{, } P_i = \{(x^i,y^i) \in [0,1]^{n_i} \times \R^{d_i} : A^i x^i + B^i y^i \le b^i\}. \label{eq:decomp}\\
		\textrm{Let } P = P_1 \times \ldots \times P_k \textrm{ denote the LP relaxation of $P^I$}. \notag
	\end{gather}
	Note that since we allow $k=1$, this also captures a general mixed-binary set. We then have the following running-time guarantee for the primal heuristic \mbw.

	\begin{theorem} \label{thm:decomp}
		Consider a feasible decomposable mixed-binary set as in equation \eqref{eq:decomp}. Let $s_i$ be such that each constraint in $P_i^I$ has at most $s_i$ binary variables, and define $c_i := \min\{ s_i \cdot (d_i + 1), n_i\}$. 		Then with probability at least $1-\delta$, \mbw with parameter $\ell=1$ returns a feasible solution within $\ln(k/\delta)\, \sum_i n_i \, 2^{n_i \log c_i}$ iterations. In particular, this bound is at most $\bar{n} k \, 2^{\bar{n} \log \bar{n}} \cdot \ln(k/\delta)$, where $\bar{n} = \max_i n_i$.
	\end{theorem}

 There are a few interesting features of this bound that indicates good properties of the proposed randomization step, apart from the fact that it is already able to find feasible solutions by itself. First, it depends on the \emph{sparsity} $s_i$ of the blocks, giving better running times on sparser problems. More importantly, the bound indicates that the algorithm works almost \emph{independently} on each of the blocks, that is, it just takes about $2^{n_i}$ iterations to find a solution for each of the blocks, instead of $2^{n_1 + \ldots + n_k}$ of a complete enumeration over the whole problem. In fact, the proof of Theorem \ref{thm:decomp} makes explicit this almost independence of the algorithm over the blocks, and motivates the uses of \emph{minimal} infeasibility certificates. Moreover, we note the important point that the algorithm is not provided the knowledge of the decomposability of the instance, it just \emph{automatically} runs ``fast'' when the problem is decomposable. This gives some indication that the proposed randomization could still exhibit good behavior on the \emph{almost decomposable} instances often found in practice (see discussion in~\cite{dey:molinaro:wang:2016}).

\paragraph{$\perturb{\ell}$ in conjunction with FP.}
 	Next, we analyze $\perturb{\ell}$ in the context of Feasibility Pump by adding it as a randomization step to the Na\"ive Feasbility Pump algorithm (Algorithm \ref{hello}); we call the resulting algorithm \WFP. This now requires understanding the complicated interplay of the randomization, rounding and projection steps: While in practice rounding and projection greatly help finding feasible solutions, their worst-case behavior is difficult to analyze and in fact they could take the iterates far away from feasible solutions. Although the general case is elusive at this point, we are nonetheless able to analyze the running time of \WFP for \emph{decomposable subset-sum} instances.

 \begin{definition}
 	A \emph{separable subset-sum set} is one of the form
 	\begin{align}
 		\{(x^1, x^2, \ldots, x^k) \in \{0,1\}^{n_1 + n_2 + \ldots + n_k}: a^i x^i = b_i ~~\forall i\} \label{eq:subset}
 	\end{align} for non-negative $(a^i,b_i)$'s.
 \end{definition}

 While this may seem like a simple class of problems, on these instances Feasibility Pump with the original randomization step from \cite{FischettiGL05} (without restarts) may not even converge, as illustrated next.

 \begin{remark}
		Consider the feasible subset-sum problem
		\begin{align*}
			\max ~& x_2\\
			 s.t. ~& 3 x_1 + x_2 = 3\\
			 & x_1, x_2 \in \{0,1\}.
		\end{align*}
		Consider the execution of the original Feasibility Pump algorithm (without restarts). The starting point is an optimal LP solution; without loss of generality, suppose it is the solution $(\frac{2}{3}, 1)$. This solution is then rounded to the point $(1, 1)$, which is infeasible. This point is then $\ell_1$-projected to the LP, giving back the point $(\frac{2}{3}, 1)$, which is then rounded again to $(1, 1)$. At this point the algorithm has stalled and applies the randomization step. Since only variable $x_2$ has strictly positive fractionality $|\frac{2}{3} - 1| = \frac{1}{3}$, only the first coordinate of $(1,1)$ is a candidate to be flipped. So suppose this coordinate  is flipped. The infeasible point $(0,1)$ obtained is then $\ell_1$-projected to the LP, giving again the point $(\frac{2}{3}, 1)$. This sequence of iterates repeats indefinitely and the algorithm does not find the feasible solution $(1,0)$.
 \end{remark}

The issue in this example is that the original randomization step never flips a variable with zero fractionality. Moreover, in Section \ref{app:3choice} of the appendix we show that even if such flips are considered, there is a more complicated subset-sum instance where the algorithm stalls.

On the other hand, we show that algorithm \WFP with the proposed randomization step always finds a feasible solution of feasible subset-sum instances, and moreover its running time again depends on the sparsity and the decomposability of the instance (in order to simplify the proof, we assume that $\wt{x} \notin P$, then $\lproj(P, \wt{x})$ is a vertex of $P$; notice that since $\lproj(P, \wt{x})$ is a linear programming problem and subset-sum instances are bounded, there is always a vertex satisfying the desired properties from $\lproj$).

\begin{theorem} \label{thm:WFP}
Consider a feasible separable subset-sum set $P$ as in \eqref{eq:subset}. 
Then with probability at least $1-\delta$, \WFP with $\ell = 2$ returns a feasible solution within $T = \lceil\ln(k/\delta)\rceil\, \sum_i n_i \, 2^{2n_i \log n_i}  \le \bar{n} k \, 2^{2 \bar{n} \log \bar{n}} \cdot \ln(k/\delta)$ iterations, where $\bar{n} = \max_i n_i$.
\end{theorem}

	To the best of our knowledge this is the first theoretical analysis of the running-time of a variant of Feasibility Pump algorithm, even for a special class of instances. As in the case of repeatedly using just $\perturb{\ell}$, the algorithm \WFP essentially works independently on each of the blocks (inequalities) of the problem, and has reduced running time on sparser instances.

The high-level idea of the proof Theorem \ref{thm:WFP} is to: 1) Show that the combination of projection plus rounding is \emph{idempotent} for these instances, namely applying them once or repeatedly yields the same effect (Lemma \ref{lemma:stabAltProj}); 2) Show that a round of randomization step plus projection plus rounding has a non-zero probability of generating an iterate closer to a feasible solution (Lemma \ref{lemma:2choice}).

\paragraph{Computational experiments.} While the analyses above give insights on the usefulness of using $\perturb{\ell}$ in the randomization step of FP, in order to attest its practical value it is important to understand how it interacts with complex engineering components present in current Feasibility Pump codes. To this end, we considered the state-of-the-art code of~\cite{FischettiS09} and modified its randomization step based on $\perturb{\ell}$. While the full details of the experiments are presented in Section \ref{sec:computation}, we summarize some of the main findings here.

We conducted experiments on MIPLIP~2010~\cite{MIPLIB2010} instances and on randomly generated two-stage stochastic models. In the first testbed there was a small but consistent improvement in both running-time and number of iterations. More importantly, the success rate of the heuristic improved consistently. In the second testbed, the new algorithm performs even better, according to all measures. It is somewhat surprising that our small modification of the randomization step could provide noticeable improvements over the code in~\cite{FischettiS09}, specially considering that it already includes several improvements over the original Feasibility Pump (e.g. constraint propagation). In addition, the proposed modification is generic and could be easily incorporated in essentially any Feasibility Pump code. Moreover, for virtually all the seeds and instances tested the modified algorithm performed better than the original version in~\cite{FischettiS09}; this indicates that, in practice, the modified randomization step dominates the previous one.

The rest of the paper is organized as follows: Section~\ref{sec:WalkSAT} we discuss and present out analysis of the proposed randomization scheme $\perturb{\ell}$, Section~\ref{sec:toyFPWalkSAT} presents the analysis of the new randomization scheme $\perturb{\ell}$ in conjunction with feasibility pump, and Section~\ref{sec:computation} describes details of our empirical experiments. 

\medskip
\noindent \textbf{Notation.}	We use $\R_+$ to denote the non-negative reals, and $[k] := \{1, 2, \ldots, k\}$. For a vector $v \in \R^n$, we use $\supp(v) \subseteq [n]$ to denote its support, namely the set of coordinates $i$ where $v_i \neq 0$. We also use $\|v\|_0 = |\supp(v)|$, and $\|v\|_1 = \sum_i |v_i|$ to denote the $\ell_1$ norm.

\section{New randomization step $\perturb{\ell}$}\label{sec:WalkSAT}

\subsection{Description of the randomization step}

	We start by describing the \wSAT algorithm~\cite{Schoning99}, that serves as the inspiration for the proposed randomization step $\perturb{\ell}$, in the context of pure-binary linear programs.	The vanilla version of \wSAT starts with a random  point $\bar{x} \in \{0,1\}^n$; if this point is feasible, the algorithm returns it, and otherwise selects any constraint violated by it. The algorithm then select a random index $i$ from the support of the selected constraint and flips the value of the entry $\bar{x}_i$ of the solution. This process is repeated until a feasible solution is obtained. It is known that this simple algorithm finds a feasible solution in expected time at most $2^n$ (see \cite{mitz} for a proof for 3-SAT instances), and Sch\"oning~\cite{Schoning99} showed that if the algorithm is restarted at every $3n$ iterations, a feasible solution is found in expected time at most a polynomial factor from $(2(1-\frac{1}{s}))^n$, where $s$ is the largest support size of the constraints.

	Based on this \wSAT algorithm, to obtain a randomization step for mixed-binary problems we are going to work on the projection onto the binary variables, so instead of looking for violated constraints we look for a \emph{certificate of infeasibility} in the space of binary variables. Importantly, we use a \textbf{minimal} certificate, which makes sure that for decomposable instances the certificate does not ``mix'' the different blocks of the problem.

Now we proceed with a formal description of the proposed randomization step $\perturb{\ell}$. Consider a mixed-binary set
\begin{gather}
	\hspace{-3pt}P^I = P \cap (\{0,1\}^n \times \R^d), \textrm{ where }
	P = \{(x,y) \in [0,1]^n \times \R^d : Ax + By \le b\}.\label{eq:MBS}
\end{gather}
We use $\projb P$ to denote the projection of $P$ onto the binary variables $x$.

\begin{definition}[Projected certificates]
 Given a mixed-binary set $P^I$ as in \eqref{eq:MBS} and a point $(\bar{x}, \bar{y}) \in \{0,1\}^n \times \R^d$ such that $\bar{x} \notin \projb P$, a \emph{projected certificate} for $\bar{x}$ is an inequality $\lambda A x + \lambda B y \le \lambda b$ with $\lambda \in \R^m_+$ such that: (i) $\bar{x}$ does not satisfy this inequality; (ii) $\lambda B = 0$. A \emph{minimal} projected certificate is one where the support of the vector $\lambda$ is minimal (i.e. the certificate uses a minimal set of the original inequalities).
\end{definition}

	Standard Fourier-Motzkin theory guarantees us that projected certificates always exist, and furthermore Caratheodory's theorem~\cite{SchrijverIntBook} guarantees that minimal projected certificates use at most $d+1$ inequalities. Together these give the following lemma.


	\begin{lemma} \label{lemma:minimalCert}
		Consider a mixed-binary set $P^I$ as in \eqref{eq:MBS} and a point $(\bar{x}, \bar{y}) \in \{0,1\}^n \times \R^d$ such that $\bar{x} \notin \projb P$. There exists a vector $\lambda \in \R^m_+$ with support of size at most $d+1$ such that $\lambda Ax + \lambda B y \le \lambda b$ is a minimal projected certificate for $\bar{x}$. Moreover, this minimal projected certificate can be obtained in polynomial-time (by solving a suitable LP).
	\end{lemma}
For completeness, see Appendix \ref{app:minPolytime} for a proof of Lemma~\ref{lemma:minimalCert}.

Now we can formally define the randomization step $\perturb{\ell}$ (notice that the condition $\lambda B = 0$ guarantees that a projected certificate has the form $ax \le b$).

\begin{algorithm}[H] \caption{$\perturb{\ell}(\bar{x})$}
	\begin{algorithmic}[1]
	\State //Assumes that $\bar{x}$ does not belong to $\projb P$
	\State Let $a x \le b$ be a minimal projected certificate for $\bar{x}$ \label{algo:wStep1}
	\State Sample $\ell$ indices from the support $\supp(a)$ uniformly and independently, let $\bI$ be the set of indices obtained  \label{algo:randCoord}
	\State (Flip coordinates) For all $i \in \bI$, set $\bar{x}_i \leftarrow 1 - \bar{x}_i$ \label{algo:wStep2}
\end{algorithmic}
\end{algorithm}

Note that in the pure-binary case and $\ell=1$, this is reduces to the main step executed during \wSAT. We remark that  the flexibility of introducing the parameter $\ell$ will be needed in Section~\ref{sec:toyFPWalkSAT}.


\subsection{Analyzing the behavior of $\perturb{\ell}$}

	In this section we consider the behavior of the algorithm $\mbw$ that tries to find a feasible mixed-binary solution by just repeatedly applying the randomization step $\perturb{\ell}$.

\begin{algorithm}[H] \caption{\mbw}
	\begin{algorithmic}[1]
	\State \textbf{input parameter:} Integer $\ell \ge 1$
	\State (Starting solution) Consider any mixed-binary point $(\bar{\x}, \bar{\y}) \in \{0,1\}^n \times \R^d$
	\Loop
			\If{$\bar{\x}$ does not belong to $\projb P$}
		    \State $\perturb{\ell}(\bar{\x})$
			\Else
				\State (Output feasible lift of $\bar{\x}$) Find $\bar{\y} \in \R^d$ such that $(\bar{\x}, \bar{\y}) \in P$, return $(\bar{\x},\bar{\y})$
			\EndIf
  	\EndLoop
\end{algorithmic}
\end{algorithm}

	As mentioned in the introduction, we show that this algorithm find a feasible solution if such exists, and the running-time improves with the sparsity and decomposability of the instance. Recall the definition of a decomposable mixed-binary problem from equation \eqref{eq:decomp}, and let $\certsup_i$ denote the maximum support size of a minimal projected certificate for the instance $P^I_i$ which consists only of the $i$th block.

\begin{theorem}[Theorem \ref{thm:decomp} restated] \label{thm:decompR}
Consider a feasible decomposable mixed-binary set as in equation \eqref{eq:decomp}. Then with probability at least $1-\delta$, \mbw with parameter $\ell=1$ returns a feasible solution within $T = \lceil\ln(k/\delta)\rceil\, \sum_i n_i \, 2^{n_i \log \certsup_i}$ iterations.
\end{theorem}

	In light of Lemma \ref{lemma:minimalCert}, if each constraint in $P_i$ has at most $s_i$ integer variables, we have $\certsup_i \le \min\{s_i \cdot (d_i+1), n_i\}$, and thus this statement indeed implies Theorem \ref{thm:decomp} stated in the introduction. We remark that similar guarantees can be obtained for general $\ell$, but we focus on the case $\ell = 1$ to simplify the exposition.

The high-level idea of the proof of Theorem \ref{thm:decompR} is the following:

	\begin{enumerate}
		\item First we show that if we run \mbw over a single block $P^I_i$, then with high probability the algorithm returns a feasible solution within $n_i\ 2^{n_i \log \certsup_i} \cdot \ln(1/\delta)$ iterations. This analysis is inspired by the one given by Sch\"oning~\cite{Schoning99} and argues that with a small, but non-zero, probability the iteration of the algorithm makes the iterate $\bar{\bx}$ closer (in Hamming distance) to a fixed solution $x^*$ for the instance.

		\medskip
		\item Next, we show that when running \mbw over the whole decomposable instance each iteration only depends on \textbf{one} of the blocks $P^I_i$; this uses the minimality of the certificates. So in effect the execution of \mbw can be split up into independent executions over each block, and thus we can put together the analysis from Item 1 for all blocks with a union bound to obtain the result.
	\end{enumerate}

	For the remainder of the section we prove Theorem \ref{thm:decompR}. We start by considering a general mixed-binary set as in equation \eqref{eq:MBS}.	Given such mixed-binary set $P^I$, we use $\certsup = \certsup(P^I)$ to denote the maximum support size of all minimal projected certificates.

	\begin{theorem} \label{thm:walkSATMI2}
		Consider the execution of \mbw over a feasible mixed-binary program as in equation \eqref{eq:MBS}. The probability that \mbw does not find a feasible solution within the first $T$ iterations is at most $(1-p)^{\lfloor T/n \rfloor}$, where $p = \certsup^{-n}$. In particular, for $T = n \cdot 2^{n \log(\certsup)} \cdot \lceil\ln(1/\delta)\rceil$ this probability is at most $\delta$ (this follows from the inequality $(1-x) \le e^{-x}$ valid for $x \ge 0$).
	\end{theorem}

	\begin{proof}
		Consider a fixed solution $x^* \in \projb P$. To analyze \mbw, we only keep track of the Hamming distance of the (random) iterate $\bar{\x}$ to $x^*$; let $\bX_t$ denote this (random) distance at iteration $t$, for $t \ge 1$. If at some point this distance vanishes, i.e. $\bX_t = 0$, we know that $\bar{\x} = x^*$ and thus $\bar{\x} \in \projb P$; at this point the algorithm returns a feasible solution for $P^I$.

	Fix an iteration $t$. To understand the probability that $\bX_t = 0$, suppose that in this iteration $\bar{\x}$ does not belong to $\projb P$, and let $a x \le b$ be the minimal projected certificate for it used in $\perturb{1}$. Since the feasible point $x^*$ satisfies the inequality $a x \le b$ but $\bar{\x}$ does not, there must be at least one index $\bi^*$ in the support of $a$ such where $x^*$ and $\bar{\x}$ differ. Then if algorithm \mbw makes a ``lucky move'' and chooses $\bI = \{\bi^*\}$ in Line \ref{algo:randCoord}, the modified solution after flipping this coordinate (the next line of the algorithm) is one unit closer to $x^*$ in Hamming distance, hence $\bX_{t+1} = \bX_t - 1$. Moreover, since $\bI$ is independent of $\bi$, the probability of choosing $\bI = \{\bi^*\}$ is $1/|\supp(a)| \ge 1/\certsup$.

		Therefore, if we start at iteration $t$ and for all the next $\bX_t$ iterations either the iterate belongs to $\projb P$ or the algorithm makes a ``lucky move'', it terminates by time $t + \bX_t$. Thus, with probability at least $(1/\certsup)^{\bX_t} \ge (1/\certsup)^n = p$ the algorithm terminates by time $t + \bX_t \le t + n$.

		To conclude the proof, let $\alpha = \lfloor T/n \rfloor$ and call iterations $i \cdot n$, \ldots, $(i+1) \cdot n -1$ the $i$-th block of iterations. If the algorithm has not terminated by iteration $i \cdot n - 1$, then with probability at least $p$ it terminates within the next $n$ iterations, and hence within the $i$-th block. Putting these bounds together for all $\alpha$ blocks, the probability that the algorithm \emph{does not} stop by the end of block $\alpha$ is at most $(1-p)^\alpha$. This concludes the proof.
	\end{proof}


	Going back to decomposable problems, we now make formal the claim that minimal projected certificates for decomposable mixed-binary sets do not mix the constraints from different blocks. Notice that projected certificates for a decomposable mixed-binary set as in equation \eqref{eq:decomp} have the form $\sum_i \lambda^i A^i x^i \le \sum_i \lambda^i b^i$ and $\lambda^i B^i = 0$ for all $i \in [k]$.

	\begin{lemma} \label{lemma:minCertificate}
		Consider a decomposable mixed-integer set as in equation \eqref{eq:decomp}. Consider a point $\bar{x} \notin \projb P$ and let $\sum_i \lambda^i A^i x^i \le \sum_i \lambda^i b^i$ be a minimal projected certificate for $\bar{x}$. Then this certificate uses only inequalities from one block $P^j$, i.e. there is $j$ such that $\lambda^i = 0$ for all $i \neq j$. Moreover, $\bar{x}^j \notin \projb P_j$.
	\end{lemma}

	\begin{proof}
		Let $\bar{x} = (\bar{x}^1, \bar{x}^2, \ldots, \bar{x}^k)$ and call the certificate $(ax \le b) \triangleq (\sum_i \lambda^i A^i x^i \le \sum_i \lambda^i b^i)$. By definition of projected certificate we have $\sum_i \lambda^i A^i \bar{x}^i > \sum_i \lambda^i b^i$, and thus by linearity there must be an index $j$ such that $\lambda^{j} A^{j} \bar{x}^{j} > \lambda^{j} b^{j}$. Moreover, as remarked earlier, decomposability implies that the certificate satisfies $\lambda^i B^i = 0$ for all $i$, so in particular for $j$. Thus, the inequality $\lambda^{j} (A^{j}, B^{j}) (x^{j}, y^{j}) \le \lambda^j b^{j}$ obtained by combining only the inequalities form $P_{j}$ is a projected certificate for $\bar{x}$. The minimality of the original certificate $ax \le b$ implies that $\lambda^i = 0$ for all $i \neq j$. This concludes the first part of the proof.

		Moreover, since $\lambda^{j} A^{j} \bar{x}^{j} > \lambda^{j} b^{j}$ and $\lambda^j B^j = 0$ we have that $\lambda^j (A^j, B^j) (\bar{x}^j, y) > \lambda^j b^j$ for all $y$, and hence $\bar{x}^j$ does not belong to $\projb P_j$. This concludes the proof.
		 \mqed
	\end{proof}

	We can finally prove the desired theorem.

	\begin{proof}[Proof of Theorem \ref{thm:decompR}.] We use the natural decomposition $\bar{\x} = (\bar{\x}^1, \ldots, \bar{\x}^k) \in \{0,1\}^{n_1} \times \ldots \times \{0,1\}^{n_k}$ of the iterates of the algorithm.	From Lemma \ref{lemma:minCertificate}, we have that for each scenario, each iteration of \mbw is associated with just one of the blocks $P^I_j$'s, namely the $P^I_j$ containing all the inequalities in the minimal projected certificate used in this iteration; let $\J_t \in [k]$ denote the (random) index $j$ of the block associated to iteration $t$. Notice that at iteration $t$, only the binary variables $x^{\J_t}$ can be modified by the algorithm.

	Let $T_i = n_i\, 2^{n_i \log n_i}\lceil \ln(k/\delta) \rceil$. Applying the proof of Theorem \ref{thm:walkSATMI2} to the iterations $\{t : \J_t = i\}$ with index $i$, we get that with probability at least $1-\frac{\delta}{k}$ the algorithm finds some $\bar{\x}^i$ in $\projb P_i$ within the first $T_i$ of these iterations. Moreover, after the algorithm finds such a point, it does not change it (that is, the remaining iterations have index $\J_t \neq i$, due to the second part of Lemma \ref{lemma:minCertificate}).

	Therefore, by taking a union bound we get that with probability at least $1 - \delta$, \emph{for all} $i \in [k]$ the algorithm finds $\bar{\x}^i \in \projb P_i$ within the first $T_i$ iterations with index $i$ (for a total of $\sum_i T_i = T$ iterations). When this happens, the total solution $\bar{\x}$ belongs to $\projb P$ and the algorithm returns. This concludes the proof.  \mqed
	\end{proof}


\section{Randomization step $\perturb{\ell}$ within Feasibility Pump}\label{sec:toyFPWalkSAT}

	In this section we incorporate the randomization step $\perturb{\ell}$ into the Na\"ive Feasibility Pump, the resulting algorithm being called \WFP. We describe this algorithm in a slightly different way and using a notation more convenient for the analysis.

Consider a mixed-binary set $P^I$ as in equation \eqref{eq:MBS}. Given a 0/1 point $\wt{x} \in \{0,1\}^n$, let $\lproj(P, \wt{x})$ denote a point $(x,y)$ in $P$ where $\|\wt{x} - x\|_1$ is as small as possible. 	Also, for a vector $v \in [0,1]^p$, we use $\round(v)$ to denote the vector obtained by rounding each component of $v$ to the closest integer; we use the convention that $\frac{1}{2}$ is rounded to 1, but any consistent rounding would suffice. Notice that operations `$\lproj$' and `$\round$' correspond precisely to Steps \ref{alg:naiveProj} and \ref{alg:naiveRound} in the Na\"ive Feasibility Pump. With this notation, algorithm \WFP can be described as follows.

	\begin{algorithm}[h] \caption{\WFP} \label{alg:WFP}
		\begin{algorithmic}[1]
    	\State \textbf{input parameter:} integer $\ell \ge 1$
    	\smallskip
    	\State Let $(\bar{x}^0, \bar{y}^0)$  be an optimal solution of the LP relaxation
    	\State Let $\wt{x}^0 = \round(\bar{x}^0)$
    	\For{t = 1,2,\ldots}
    		\State $(\bar{\x}^t, \bar{\y}^t) = \lproj(P, \tx^{t-1})$ \label{alg:lproj}
    		\State $\tx^t = \round(\bar{\x}^t)$ \label{alg:round}

				\smallskip
				\If{$(\tx^t, \bar{\y}^t) \in P$} \Comment{equivalently, $\tx^t \in \projb(P)$}
					\State Return $(\tx^t, \bar{\y}^t)$ \label{algo:WFPRet}
				\EndIf

				\smallskip
				\If{$\tx^t = \tx^{t-1}$} \Comment{iterations have stalled}
					\State $\tx^t = \perturb{\ell}(\tx^t)$
				\EndIf
	  	\EndFor
    \end{algorithmic}
  \end{algorithm}

Note that stalling in the above algorithm is determined using the condition $\tx^t = \tx^{t-1}$. What about `long cycle' stalling, that is $\tx^t = \tx^{t'}$ where $t' < t -1$, but $\tx^{t'}, \dots, \tx^{t -1}$ are all distinct binary vectors. As it turns out (assuming no numerical errors) a consistent rounding rule implies that stalling will always occur with cycles of length two. 

\begin{theorem}\label{thm:sidenote}
With consistent rounding, long cycles cannot occur. 
\end{theorem}
We present a proof of \ref{thm:sidenote} in Appendix \ref{sec:side}.
For the remainder of the section, we analyze the behavior of algorithm \WFP on separable subset-sum instances, proving Theorem \ref{thm:WFP} stated in the introduction.

	\subsection{Running time of \WFP for separable subset-sum instances: Proof of Theorem \ref{thm:WFP}}

	Notice that the projection operators `$\lproj$' and `$\round$' now present also act on each block independently, namely  given a point $x = (x^1, \ldots, x^k) \in \R^{n_1} \times \ldots \times \R^{n_k}$, if $(\check{x}^1, \ldots, \check{x}^k) = \lproj(P, x)$ then $\check{x}^i = \lproj(P_i, x^i)$ for all $i \in [k]$, and similarly for `$\round$'.
	Therefore, as in the proof of Theorem \ref{thm:decompR}, it suffices to analyze the execution of algorithm \WFP over a single block/inequality of the separable subset-sum problem. More precisely, it suffices to prove the following guarantee for $\WFP$ on a general subset-sum instance.

  \begin{theorem} \label{thm:WFPGen}
		Consider a feasible subset-sum problem $P \subseteq \R^n$. Then for every $T \ge 1$, the probability that \WFP with $\ell = 2$ does not find a feasible solution within the first $2T$ iterations is at most $(1-p)^{\lfloor T/n \rfloor}$, where $p = (1/n^2)^n$. In particular, for $T = n \cdot 2^{2n \log n} \cdot \lceil\ln(1/\delta)\rceil$ this probability is at most $\delta$.
	\end{theorem}

 	The high-level idea of the proof of this theorem is the following. We use a similar strategy as before, where we consider a fixed feasible solution $x^*$ and track its distance to the iterates $\wt{\bx}^t$ generated by algorithm \WFP. However, while again the randomization step $\perturb{2}$ brings $\wt{\bx}^t$ closer to $x^*$ with small but non-zero probability, the issue is that the projections `$\lproj$' and `$\round$' in the next iterations could send the iterate even further from $x^*$. To analyze the algorithm we then use the structure of subset-sum instances to: 1) First control the combination `$\lproj + \round$' in Steps \ref{alg:lproj} and \ref{alg:round}, showing that in this case they are \emph{idempotent}, namely applying them once or repeatedly yields the same effect (Lemma \ref{lemma:stabAltProj}); 2) Strengthen the analysis of Theorem \ref{thm:decompR} to show that a round of $\perturb{2}$ \emph{plus} `$\lproj + \round$' still has a non-zero probability of generating a point closer to $x^*$ (Lemma \ref{lemma:2choice}). For this, it will be actually important that we use $\ell = 2$ in algorithm \WFP (actually $\ell \ge 2$ suffices).

	For the remainder of the section we prove Theorem \ref{thm:WFPGen}. To simplify the notation we omit the polytope $P$ from the notation of $\lproj$. We assume that our subset-sum problem $P = \{x \in [0,1]^n : ax = b\}$ is such that \emph{all} coordinates of $a$ are positive, since components with $a_i = 0$ do not affect the problem (more precisely, after the first iteration of the algorithm, the value of $\wt{\bx}^t_i$ is set to 0 or 1 and does not change anymore, and this value does not affect the feasibility of the solutions $\wt{\bx}^t$'s). Also remember that subset-sum problems only have binary variables.

	Given a point $\wt{x} \in \{0,1\}^n$, let $\altProj(\wt{x}) \in \{0,1\}^n$ be the effect of applying to $\wt{x}$ $\lproj(.)$ and then $\round(.)$. Notice that if $\wt{x}$ belongs to $P$, then $\altProj(\wt{x}) = \wt{x}$. Then algorithm \WFP can be thought as performing a $\altProj$ operation, then checking if the iterate obtained either belongs to $P$ (in which case it exits) of if it equals the previous iterate (in which case it applies $\perturb{2}$); if neither of these occur, then another $\altProj$ operation is performed. So an important component for analyzing this algorithm is getting a good control over a sequence of $\altProj$ operations. For that, define the iterated operation  $\altProj^t(\wt{x}) = \altProj\left(\altProj^{t-1}(\wt{x})\right)$ (with $\altProj^1 = \altProj$) and if the sequence $(\altProj^t(\wt{x}))$ stabilizes at a point, let $\altProj^*(\wt{x})$ denote this point.



	A crucial observation, given by the next lemma, is that for subset-sum instances the operation of $\altProj$ is idempotent, namely it stabilizes after just one operation.

	\begin{lemma} \label{lemma:stabAltProj}
		Let $P$ be a subset-sum instance. Then for every $\wt{x} \in \{0,1\}^n$, $\altProj^*_P(\wt{x}) = \altProj_P(\wt{x})$.
	\end{lemma}

	 \begin{proof}
 Again to simplify the notation we omit the polyhedron $P$ when writing $\lproj$ and $\altProj$. Let $\bar{x}=\lproj(\wt{x})$ and recall it is an extreme point of $P$. Clearly, if $\wt{x} \in P$ then $\altProj(\wt{x})=\wt{x}$ and hence $\altProj^*(\wt{x}) = \altProj(\wt{x})$. Similarly, if $\bar{x}$ is a 0/1 point then $\altProj(\wt{x})=\bar{x}$, and again $\altProj^*(\wt{x}) = \altProj(\wt{x})$.

	 Thus, assume that $\wt{x}\notin P$ and $\bar{x}$ is not a 0/1 point. Since $\bar{x}$ is an extreme point of the subset-sum LP $P$ it has exactly 1 fractional coordinate, so by permuting indices we assume without loss of generality:
\begin{enumerate}
\item $\bar{x}_1 = \dots = \bar{x}_k = 1$.
\item $\bar{x}_{k + 1} \in (0, \ 1)$.
\item $\bar{x}_{k + 2}= \dots = \bar{x}_n = 0$
\item $a_{k + 2} \geq a_{k + 3}\geq \dots \geq a_{n}$.
\item $a_1 \leq a_2 \leq a_3 \leq \dots \leq a_{k}$.
\end{enumerate}

	Now we look at the points obtained after applying $\round(.)$ and $\lproj(.)$ to $\bar{x}$, namely let $\wt{x}' := \round(\bar{x}) = \altProj(\wt{x})$ and let $\bar{x}' := \lproj(\wt{x}')$. Notice that $\bar{x}'$ is obtained by solving:
\begin{eqnarray}
\textup{min}~&\sum_{ \{j \,|\, \wt{x}'_j = 0\}}x_j + \sum_{ \{j \,|\, \wt{x}'_j = 1\}}(1 - x_j) \notag\\
\textup{s.t.}~& ax=b \label{eq:lproj}\\
&\ 0\le x\le 1.\notag
\end{eqnarray}

\medskip
\noindent \textbf{Case 1:} $\bar{x}_{k+1}<1/2$. Then $\wt{x}'_i=1$ for all $i \le k$, $\wt{x}'_i =0$ for all $i \ge k+1$; also notice $\wt{x}'\le \bar{x}$, and hence $a\wt{x}'<b$; thus $\bar{x}'$ is obtained from $\wt{x}$' by increasing some components of 0 value. We have three subcases:
\begin{enumerate}
\item[a.] If $a_{k+1}>a_{k+2}$: then $a_{k+1}$ is the largest coordinate of $a$ where $\wt{x}'$ has value 0, so it follows from \eqref{eq:lproj} that $\bar{x}'$ is obtained from $\wt{x}'$ by raising its $(k+1)$-component from 0 to $\bar{x}_{k+1}$. Thus, $\bar{x}' = \bar{x}$, and hence $\altProj(\altProj(\wt{x})) = \round(\bar{x}')$ equals $\round(\bar{x}) = \altProj(\wt{x})$; this implies $\altProj^*(\wt{x}) = \altProj(\wt{x})$.

\medskip
\item[b.] If $a_{k+1} < a_{k+2}$: then $\bar{x}'$ is obtained from $\wt{x}$ by raising its $(k+2)$-component to a value that is at most $\bar{x}_{k+1} < 1/2$. Now, $\round(\bar{x}')=\wt{x}'$, so again we get $\altProj(\altProj(\wt{x})) = \round(\bar{x}') = \wt{x}' = \altProj(\wt{x})$ and we are done.

\medskip
\item[c.] If $a_{k+1} = a_{k+2}$: Since $\bar{x}'$ is a vertex of the subset-sum LP $P$, again it only has 1 fractional component (either $k+1$ or $k+2$) and then it is easy to see that $\bar{x}'$ is equal to the one in either Case (a) or Case (b) above; thus the result also holds for this case.
\end{enumerate}

\noindent \textbf{Case 2:} $\bar{x}_{k+1} \ge 1/2$. Then $\wt{x}'$ is such that $\wt{x}'_i=1$ for all $i \le k+1$ and $\wt{x}'=0$ for all $i \ge k+2$; also notice $\wt{x}' \ge \bar{x}$ and hence $a\wt{x}'>b$.
Now, consider $\bar{x}'=\lproj(\wt{x}')$:
\begin{enumerate}
\item[a.] If $a_{k} < a_{k+1}$: This is analogous to Case 1a: $\bar{x}'$ is obtained by lowering the $(k+1)$-coordinate of $\wt{x}'$ from 1 to $\bar{x}$, and thus $\bar{x}' = \bar{x}$; the rest of the proof is identical to Case 1a.

\medskip
\item[b.] If  $a_{k} > a_{k+1}$: In this case, $\bar{x}'$ is obtained by lowering the $k$-component of $\wt{x}'$. Since $a\bar{x}=a\bar{x}'=b$, and $k$ and $(k+1)$ are the only components where $\bar{x}$ and $\bar{x}'$ differ, we have: $a_k  + a_{k+1} \bar{x}_{k+1} = a_k \bar{x}'_k + a_{k+1} $. Hence $\bar{x}'_k = 1-\frac{a_{k+1}}{a_k}(1-\bar{x}_{k+1})\ge 1/2$ and $\round(\bar{x}') = \wt{x}'$; the rest of the proof is identical to Case 1b.

\medskip
\item[c.] If  $a_{k} = a_{k+1}$: Identical to Case 1c.
\end{enumerate} \mqed
\end{proof}

	Therefore, there is not much loss in looking at a ``compressed'' version of algorithm \WFP that packs repeated applications of $\altProj$ until stalling happens into a single $\altProj^*$; more formally, we have the following algorithm (stated in the pure-binary case to simplify the notation).

		\begin{algorithm}[h] \caption{\WFP-Compressed} \label{alg:WFPComp}
		\begin{algorithmic}[1]
    	\State \textbf{input parameter:} integer $\ell \ge 1$
    	\smallskip
    	\State Let $\bar{x}^0$  be an optimal solution of the LP relaxation
    	\State Let $\wt{z}^0 = \round(\bar{x}^0)$
    	\For{$\tau$ = 1,2,\ldots}
    		\State $\bar{\bz}^{\tau} = \altProj^*(\wt{\bz}^{\tau - 1})$ \label{alg:lprojComp}

				\smallskip
				\If{$\wt{\bz}^{\tau} \in P$}
					\State Return $\wt{\bz}^{\tau}$ \label{algo:WFPRetComp}
				\EndIf

				\smallskip

				\State $\wt{\bz}^{\tau} = \perturb{\ell}(\wt{\bz}^{\tau})$
			\EndFor
    \end{algorithmic}
  \end{algorithm}

	Intuitively, Lemma \ref{lemma:stabAltProj} should imply that packing the repeated applications of $\altProj$ into a single $\altProj^*$ should not save more than 1 iteration. To see this more formally, assume that both algorithms use as starting point the same optimal solution of the LP, so $\wt{z}^0 = \wt{x}^0$. Now condition on a scenario where we have $\wt{\bz}^{\tau} = \wt{\x}^t$ at the beginning of iterations $\tau$ and $t$ of algorithms \WFP-Compressed and \WFP respectively (for $\tau, t \ge 1$). Then we claim that either both algorithms return at the current iteration, or $\wt{\bz}^{\tau + 1}$ has the same distribution as either $\wt{\x}^{t + 1}$ or $\wt{\x}^{t + 2}$ (at the beginning of they respective iterations): If $\wt{\bz}^{\tau} = \wt{\x}^t \in P$, then both algorithms return; if $\wt{\x}^t \notin P$ but $\wt{\x}^t = \wt{\x}^{t-1}$, then both algorithms \WFP-Compressed and \WFP employ $\perturb{2}$ over $\wt{\bz}^{\tau}=\wt{\x}^t$, in which case $\wt{\bz}^{\tau + 1}$ has the same distribution as $\wt{\x}^{t + 1}$; finally, if $\wt{\x}^t \neq \wt{\x}^{t-1}$, then \WFP at the beginning of the next iteration will have $\wt{\x}^{t+1} = \altProj(\wt{\x}^t)$, which by Lemma \ref{lemma:stabAltProj} (and $t \ge 1$) equals $\wt{\x}^t$ itself, and so it will employ $\perturb{2}$ to $\wt{\x}^{t+1} = \wt{\x}^t$ and again we have that $\wt{\x}^{t+2}$ has the same distribution as $\wt{\bz}^{\tau + 1}$.

	Therefore, since we can employ this argument to couple iterations $\le \tau$ of \WFP-Compressed with iterations $\le 2\tau$ of \WFP, we have the following result.

	\begin{lemma} \label{lemma:coupling}
		Consider the application of algorithms \WFP and \WFP-Compressed over the subset-sum problem $P$. Then the probability that algorithm \WFP returns after at most $2T$ iterations is at least the probability that algorithm \WFP-Compressed after at most $T$ iterations.
	\end{lemma}

	Therefore, it suffices to upper bound the number of iterations of \WFP-Compressed until it returns. To avoid ambiguity, let $\bz^{\tau}$ be the value of $\wt{\bz}^{\tau}$ at the \emph{beginning} of iteration $\tau$ of \WFP-Compressed. Notice that $z^1 = \altProj^*(\wt{x}^0)$, and $\bz^{\tau+1} = \altProj^*(\perturb{2}(\bz^\tau))$ for $\tau \ge 2$. It suffices to show that with probability at least $1 - (1-p)^{T/n}$, there is $\tau \le T/2$ such that $\bz^\tau$ belongs to $P$.

	To do so, for $\wt{x} \in \{0,1\}^n$ and $I \subseteq [n]$ let $\flip(\wt{x}, I)$ denote the 0/1 vector obtained starting from $\wt{x}$ and flipping the value of all coordinates that belongs to $I$. Notice that (up to scaling) the only possible projected certificates for our subset-sum problem are $ax \ge b$ and $ax \le b$.
	 Since we have assumed that the vector $a$ has full support, it follows that on this problem $\perturb{2}(\wt{x}) = \flip(\wt{x}, \bI)$ for $\bI$ being the set obtained by sampling independently two indices uniformly from $[n]$.

	The next lemma then shows that there is always a ``lucky choice'' of set $\bI$ in $\perturb{2}(\bz^\tau)$ that brings $\bz^{\tau + 1} = \altProj^*(\perturb{2}(\bz^\tau))$ closer to a fixed solution $x^*$ to the subset-sum problem.

The following definition is convenient.
\begin{definition}\label{defn:stall}
A point $\wt{x} \in \{0, 1\}^n$ is called a stalling solution if $\altProj(\wt{x}) = \wt{x}$.
\end{definition}

	\begin{lemma} \label{lemma:2choice}
		Let $x^* \in \{0,1\}^n$ be a feasible solution to the subset-sum problem. Consider $\wt{x} \in \{0,1\}^n$ with $a \wt{x} \neq b$ that satisfies the fixed point condition $\altProj(\wt{x}) = \wt{x}$. Then  there is a set $I \subseteq [n]$ of size at most 2 such that the point $x' = \altProj^*_P(\flip(\wt{x}, I))$ is closer to $x^*$ than $\wt{x}$, namely $\|x' - x^*\|_0 \le \|\wt{x} - x^*\|_0 - 1$.
	\end{lemma}
	\begin{proof}
		Again to simplify the notation we omit $P$ from $\lproj$ and $\altProj$, and use $\flip(\wt{x}, j)$ instead of $\flip(\wt{x}, \{j\})$ in the singleton case. 


We start with a couple of claims.
\paragraph{Claim 1} 		Suppose $\wt{x} \in \{0,1\}^n$ is a stalling point. If $a \wt{x} < b$, then there is $k \notin \supp(\wt{x})$ such that  $\lproj(\wt{x})_i = \wt{x}_i$ for all $i \neq k$, and  $\lproj(\wt{x})_k \in (0,\frac{1}{2})$. Similarly, if $a \wt{x} > b$, then there is $k \in \supp(\wt{x})$ such that $\lproj(\wt{x})_i = \wt{x}_i$ for all $i \neq k$, and $\lproj(\wt{x})_k \in [\frac{1}{2}, 1)$.

	\begin{proof}[Proof of Claim 1]
		We only prove the first statement, the proof of the second is completely analogous. Since $\wt{x}$ is stalling we have that $\round(\lproj(\wt{x})) = \wt{x}$, and since $\lproj(\wt{x})$ is an extreme point of the subset-sum problem $P$ it has at most 1 fractional component, and hence only differs in one component $k$ from $$\round(\lproj(\wt{x})) = \wt{x}.$$ Since $a \cdot \lproj(\wt{x}) = b > a \cdot \wt{x}$, we have that $\wt{x}_k = 0$ and $\lproj(\wt{x})_k > 0$; since $\round(\lproj(\wt{x})_k) = \wt{x}_k = 0$, we have $\lproj(\wt{x})_k < \frac{1}{2}$.
	\end{proof}

\paragraph{Claim 2}
		Consider a point $\wt{x} \in \{0,1\}^n$.
		\begin{enumerate}
			\item If the objective value of \eqref{eq:lproj} is strictly less than $\frac{1}{2}$, then $\altProj(\wt{x}) = \wt{x}$.
			\item If the objective value of \eqref{eq:lproj} is strictly less than 1, then $\|\altProj(\wt{x}) - \wt{x}\|_0 \le 1$.
		\end{enumerate}

	\begin{proof}[Proof of Claim 2]
		Let $\bar{x} = \lproj(\wt{x})$ be an optimal solution for \eqref{eq:lproj}. Proof of Part 1: the assumption implies that $|\bar{x}_i - \wt{x}_i| < \frac{1}{2}$ for all $i$, which directly implies that $\altProj(\wt{x}) = \round(\bar{x}) = \wt{x}$.

		Proof of Part 2: the assumption implies that there can be at most one index $j$ with $|\bar{x}_j - \wt{x}_j|\geq \frac{1}{2}$, which implies that for all $i \neq j$, $\altProj(\wt{x})_i = \round(\bar{x}_i) = \wt{x}_i$ and the result follows. \mqed
	\end{proof}

Now we are ready to present the proof of Lemma \ref{lemma:2choice}. Let $x^*$ and $\tilde{x}$ be as in the statement of the Lemma. 
	From Lemma \ref{lemma:stabAltProj} we know that $$\altProj^*(\flip(\wt{x}, J)) = \altProj(\flip(\wt{x}, J)),$$ so it suffices to work with the right-hand side instead.	Since $\wt{x}\neq x^*$
	we have $\supp(\wt{x}) \neq \supp(x^*)$. We separate the proof in three cases depending on the relationship between these supports.

\medskip \noindent \textbf{Case 1:} $\supp(\wt{x}) \subsetneq \supp(x^*)$: Pick any $j \in \supp(x^*) \setminus \supp(\wt{x})$ and notice that $\|\flip(\wt{x},j) - x^*\|_0 = \|\wt{x} - x^*\|_0 - 1$. Notice that both $\supp(\wt{x})$ and $\supp(\flip(\wt{x},j))$ are
contained in the support of $x^*$, and hence we have $a \wt{x} \le b$ and $a \cdot \flip(\wt{x},j) \le b$. Moreover, since $\flip(\wt{x},j) \ge \wt{x}$, it is easy to see that the optimal value of \eqref{eq:lproj} for $\flip(\wt{x}, j)$ is \emph{strictly less than} that for $\wt{x}$ (we need to raise fewer variables to make the point satisfy $ax = b$), which by Claim 1 is at most $\frac{1}{2}$. Thus, employing Part 1 of Claim 2 to $\flip(\wt{x},j)$ gives that $\altProj(\flip(\wt{x},j)) = \flip(\wt{x},j)$, which is the desired point closer to $x^*$.

\medskip \noindent \textbf{Case 2:} $\supp(x^*) \subsetneq \supp(\wt{x})$: The proof is the same as above, with the only change that we take $j \in \supp(\wt{x}) \setminus \supp(x^*)$.

\medskip \noindent \textbf{Case 3:} The supports $\supp(x^*)$ and $\supp(\wt{x})$ are not contained in one another. In this case $a\wt{x}$ can be either $< b$ or $>b$:
\begin{enumerate}
\item If $a\wt{x}< b$. Take $m\in \supp(x^*)\setminus \supp(\wt{x})$. If $a\cdot \flip(\wt{x},m)\le b$, then we can argue exactly as in Case 1 to get that $\altProj(\flip(\wt{x},m)) = \flip(\wt{x},m)$, which is closer to $x^*$ than $\wt{x}$. So consider the case $a \cdot \flip(\wt{x},m) > b$. Take $i \in \supp(\wt{x})\setminus \supp(x^*)$ and consider $\flip(\wt{x}, \{m,i\})$, which is 2 units closer to $x^*$ in Hamming distance.

We claim that the optimal value of \eqref{eq:lproj} for $\flip(\wt{x}, \{m,i\})$ is strictly less than 1. Suppose $a \cdot \flip(\wt{x}, \{m,i\}) \le b$; since $a \cdot \flip(\wt{x}, m) > b$ (notice $\flip(\wt{x}, m)$ is obtained from $\flip(\wt{x}, \{m,i\})$ by increasing coordinate $i$ to 1), this means that we can make $\flip(\wt{x}, \{m,i\})$ satisfy $ax = b$ by increasing coordinate $i$ to a value \emph{strictly less} than 1, thus upper bounding the optimum of \eqref{eq:lproj}. On the other hand, consider $a \cdot \flip(\wt{x}, \{m,i\}) > b$; notice $a \cdot \flip(\wt{x}, i) \le a \cdot \wt{x} < b$ (the last uses a running assumption), and thus again we can make $\flip(\wt{x}, \{m,i\})$ satisfy $ax = b$ by decreasing coordinate $m$ to a value strictly smaller than 1. This proves the claim.

With this claim in place, we can just employ Part 2 of Claim 2 to $\flip(\wt{x}, \{m,i\})$ and triangle inequality to obtain that $\|\altProj(\flip(\wt{x}, \{m,i\})) - x^*\|_0$ is at most $$1 + \|\flip(\wt{x}, \{m,i\}) - x^*\|_0 = 1 + \|\wt{x} - x^*\|_0 - 2,$$ which gives the desired result.

\item If $a\bar{x}> b$. The proof of this case mirrors that of the above case (only with the inequalities $<$ and $>$ reversed throughout).
\end{enumerate}
\mqed
	\end{proof}

	Notice that since $\bz^\tau$ is obtained from $\altProj^*(.)$, it satisfies the fixed point condition $\altProj(\bz^\tau) = \bz^\tau$. Thus, as long as $\bz^{\tau}$ does not belong to $P$ we can apply the above lemma to obtain that with probability at least $\frac{1}{n^2}$ we have $\bI$ in $\perturb{2}$ equal to the set $I$ in the lemma and thus the iterate moves closer to a feasible solution; more formally we have the following.

	\begin{corollary} \label{cor:2choice}
		Let $x^* \in \{0,1\}^n$ be a feasible solution to the subset-sum problem $P$. Then $$\Pr\Big(\|\bz^{\tau + 1} - x^*\|_0 \le \|\bz^\tau - x^*\|_0 - 1 ~\Big\vert~ \bz^{\tau} \notin P \Big) \ge \frac{1}{n^2}.$$
	\end{corollary}

	Now we can conclude the proof of Theorem \ref{thm:WFPGen} arguing just like in the proof of Theorem \ref{thm:walkSATMI2}.

	\begin{proof}[Proof of Theorem \ref{thm:WFPGen}]
		Consider $x^* \in P$ and let $\bZ_\tau = \|\bz^{\tau} - x^*\|_0$. Notice that $\bZ_\tau = 0$ implies $\bz^\tau = x^*$ and hence $\bz^\tau \in P$. Corollary \ref{cor:2choice} gives that $\Pr(\bZ_{\tau + 1} \le \bZ_{\tau} - 1 \mid \bz^\tau \notin P) \ge \frac{1}{n^2}$. Therefore, if we start at iteration $\tau$ and for all the next $\bZ_{\tau}$ iterations either the iterate $\bz^{\tau'}$ belongs to $P$ or the algorithm reduces $\bZ_{\tau'}$, it terminates by time $\tau + \bZ_{\tau}$. Thus, with probability at least $(1/n^2)^{\bZ_{\tau}} \ge (1/n^2)^n = p$ the algorithm terminates by time $t + \bZ_{\tau} \le t + n$.

		To conclude the proof, let $\alpha = \lfloor T/n \rfloor$  and call time steps $i \cdot n$, \ldots, $(i+1) \cdot n -1$ the $i$-th block of time. From the above paragraph, the probability that there is $\tau$ in the $i$th block of time such that $\bz^\tau \in P$ conditioned on $\bz^{i \cdot n - 1} \notin P$ is at least $p$. Using the chain rule of probability gives that the probability that there is no $\bz^\tau \in P$ within any of the $\alpha$ blocks is at most $(1-p)^\alpha$. This concludes the proof.  \mqed
	\end{proof}


\section{Computations}\label{sec:computation}

In this section, we describe the algorithms that we have implemented and report computational experiments comparing the performance of the original Feasibility Pump 2.0 algorithm from~\cite{FischettiS09}, which we denote by \FPorig, to our modified code that uses the new perturbation procedure. The code is based on the current version of the Feasibility Pump 2.0 code (the one available on the NEOS servers), which is implemented in C++ and linked to IBM ILOG CPLEX 12.6.3~\cite{CPLEX} for preprocessing and solving LPs. All features such as constraint propagation which are part of the Feasibility Pump 2.0 code have been left unchanged.

All algorithms have been run on a cluster of identical machines, each equipped with an Intel Xeon CPU E3-1220 V2 running at 3.10GHz and 16 GB of RAM. Each run had a time limit of half an hour.


\subsection{WalkSAT-based perturbation}

In preliminary tests, we implemented the algorithm \WFP as described in the previous section. However, its performance was not competitive with \FPorig. In hindsight, this can be justified by the following reasons:
\begin{itemize}
\item Picking a fixed $\ell$ can be tricky. Too small or too big a value can lead to slow convergence in practical implementations.
\item Using $\perturb{\ell}$ at each perturbation step can be overkill, as in most cases the original perturbation scheme does just fine.
\item Computing the minimal certificate is too expensive, as it requires solving LPs.
\end{itemize}

For the reasons above, we devised a more conservative implementation of a perturbation procedure inspired by \wSAT, which we denote by \WFPbase. The algorithm works as follows. Let $F\subset [n]$ be the set of indices with positive fractionality $|\wt{x}_j - \bar{x}_j|$. If $TT \le |F|$, then the perturbation procedure is just the original one in \FPorig.
Else, let $S$ be the union of the supports of the constraints that are not satisfied by the current point $(\wt{x}, \bar{y})$.
We select the $|F|$ indices with largest fractionality $|\wt{x}_j - \bar{x}_j|$ and select uniformly at random $\textup{min}\{|S|, TT-|F|\}$ indices from $S$, and flip the values in $\wt{x}$ for all the selected indices.

Note also that the above procedure applies only to the case in which a cycle of length one is detected. In case of longer cycle, we use the very same restart strategy of \FPorig.

\subsection{Computational results}

We tested the two algorithms on two classes of models: two-stage stochastic models, and the MIPLIB 2010 dataset.


\paragraph{Two-stage stochastic models.} In order to validate the hypothesis suggested by the theoretical results that our walkSAT-based perturbation should work well on almost-decomposable models, we tested \WFPbase on two-stage stochastic models. These are the deterministic equivalent of two-stage stochastic programs and have the form 
\begin{align*}
	&Ax + D^i y^i \le b^i ~~, i \in \{1, \ldots, k\}\\
	&x \in \{0,1\}^p\\
	&y^i \in \{0,1\}^q ~~, i \in \{1, \ldots, k\}.
\end{align*}
 The variables $x$ are the first-stage variables, and $y^i$ are the second-stage variables for the $i$th scenario. Notice that these second-stage variables are different for each scenario, and are only coupled through the first-stage variables $x$. Thus, as long as the number of scenarios is reasonably large compared to dimensions of $x, y^1, \ldots, y^k$, these problems are to some extent almost-decomposable.

	For our experiments we randomly generated instances of this form as follows: (1) the entries in $A$ and the $D^i$'s are independently and uniformly sampled from $\{-10, \ldots, 10\}$; (2) to guarantee feasibility, a 0/1 point is sampled uniformly at random from $\{0,1\}^{p + k \cdot q}$ and the right-hand sides $b^i$ are set to be the smallest ones that make this points feasible. We generated 50 instances, 5 for each setting of parameters $k=\{5,15,25,35,45\}$, $p = \{10,20\}$, $q = 10$.
	
	We compared the two algorithms \FPorig and \WFPbase over these instances using ten different random seeds. A seed by seed comparison is reported in Table~\ref{tab:stoch}. In the tables, \texttt{\#found} denotes the number of models for which a feasible solution was found, while \texttt{time} and \texttt{itr.} report the shifted geometric means~\cite{Achterberg07} of running times and iterations, respectively.

\begin{table}[ht]
\begin{center}
\begin{tabular}{lrrrrrr}
\toprule
& \multicolumn{2}{c}{\# found} & \multicolumn{2}{c}{time (s)} & \multicolumn{2}{c}{itr.}\\
\cmidrule{1-7}
Seed & \FPorig & \WFPbase & \FPorig & \WFPbase & \FPorig & \WFPbase \\
\midrule
  1 &  28 & \textbf{31} &  4.12 & \textbf{3.36} &  124.43 & \textbf{76.02} \\
  2 &  26 & \textbf{35} &  4.06 & \textbf{3.17} &  122.51 & \textbf{82.85} \\
  3 &  25 & \textbf{37} &  4.00 & \textbf{3.02} &  117.74 & \textbf{72.50} \\
  4 &  26 & \textbf{36} &  4.28 & \textbf{3.40} &  119.82 & \textbf{75.17} \\
  5 &  25 & \textbf{31} &  4.20 & \textbf{3.44} &  124.41 & \textbf{81.66} \\
  6 &  26 & \textbf{35} &  3.98 & \textbf{3.56} &  122.74 & \textbf{79.73} \\
  7 &  25 & \textbf{27} &  4.22 & \textbf{3.98} &  126.77 & \textbf{91.59} \\
  8 &  28 & \textbf{38} &  3.82 & \textbf{3.10} &  112.91 & \textbf{73.92} \\
  9 &  25 & \textbf{31} &  4.22 & \textbf{3.67} &  117.61 & \textbf{83.46} \\
 10 &  25 & \textbf{32} &  4.12 & \textbf{3.57} &  116.92 & \textbf{88.23} \\
\bottomrule
\end{tabular}
\caption{Aggregated results on two-stage stochastic models.}
\label{tab:stoch}
\end{center}
\end{table}

	Notice that \WFPbase performed substantially better than \FPorig, in agreement with our theoretical results. Using the walkSAT-based perturbation the average number of successful instances increased by $28\%$, while average runtime was reduced by $17\%$ and average number of iterations was reduced by $33\%$.


\paragraph{MIPLIB 2010.}	We also compared the algorithms on a subset of models from MIPLIB 2010~\cite{MIPLIB2010}.
The subset is defined by the models for which at least one of the two algorithms took more than 20 iterations to find a feasible solution (if any); the remaining models are basically too easy and not useful for comparing the two perturbation procedures. We are thus left with a subset of 82 models. Again we compared the two algorithms using ten different random seeds. A seed by seed comparison is reported in Table~\ref{tab:mip2010}.

	Even though the improvement in this heterogeneous testbed was less dramatic as in the two-stage stochastic models, as expected, \WFPbase still consistently dominates \FPorig: it can find more solutions in 7 out 10 cases (in the remaining 3 cases it is a tie), taking always less time and almost always fewer iterations. On average over the seeds, \WFPbase increased the number of successfully solved instances by $6\%$, reduced by the computation time by $8.4\%$ and reduced the number of iterations by $5.9\%$. 

	In conclusion, given that the suggested modification is very simple to implement, and appears to dominate \FPorig consistently, it suggests it is a good idea to add it as a feature in all future feasibility pump codes. 
	

\begin{table}[ht]
\begin{center}
\begin{tabular}{lrrrrrr}
\toprule
& \multicolumn{2}{c}{\# found} & \multicolumn{2}{c}{time (s)} & \multicolumn{2}{c}{itr.}\\
\cmidrule{1-7}
Seed & \FPorig & \WFPbase & \FPorig & \WFPbase & \FPorig & \WFPbase \\
\midrule
  1 &  33 & \textbf{34} &  1070.35 & \textbf{1068.09}  & \textbf{103.38} & 104.59 \\
  2 &  \textbf{34} & \textbf{34} &  1073.03 & \textbf{1004.84}  & 108.65 & \textbf{104.05} \\
  3 &  34 & \textbf{39} &  1125.44 &  \textbf{976.16}  & 107.10 &  \textbf{96.18} \\
  4 &  34 & \textbf{36} &  1045.10 &  \textbf{976.31}  & 101.30 &  \textbf{96.24} \\
  5 &  31 & \textbf{32} &  1033.60 &  \textbf{974.56}  &  96.67 &  \textbf{94.36} \\
  6 &  \textbf{34} & \textbf{34} &   974.47 &  \textbf{880.05}  &  99.61 &  \textbf{91.20} \\
  7 &  33 & \textbf{36} &   972.96 &  \textbf{877.45}  & 102.39 &  \textbf{95.04} \\
  8 &  29 & \textbf{32} &  1085.82 & \textbf{1049.22}  & 104.63 & \textbf{103.22} \\
  9 &  \textbf{37} & \textbf{37} &  1065.50 &  \textbf{937.19}  & 101.44 &  \textbf{91.73} \\
 10 &  32 & \textbf{37} &  1096.99 &  \textbf{913.50}  & 103.01 &  \textbf{90.85} \\
\bottomrule
\end{tabular}
\caption{Aggregated results on MIPLIB2010.}
\label{tab:mip2010}
\end{center}
\end{table}


\section*{Acknowledgments}

	We would like to thank Andrea Lodi for discussions and clarifications on Feasibility Pump. Santanu S. Dey and Andres Iroume would like to gratefully acknowledge the support of NSF grants CMMI 1562578 and CMMI 1149400 respectively.

\bibliographystyle{alpha}
\bibliography{test}


	\newpage
	\appendix
	\noindent {\LARGE \textbf{Appendix}}

\section{Minimal projected certificates can be found in polynomial time} \label{app:minPolytime}

Consider the following LP:
\begin{eqnarray*}
\textup{max}& \lambda A \bar{x} - \lambda b \\
\textup{s.t.} & B^T \lambda = 0 \\
&e^T \lambda = 1\\
& \lambda \geq 0,
\end{eqnarray*}
where $e$ is the all-ones vector. Since we assumed a projected certificate exists, this LP is feasible and has strictly positive optimal value.

An optimal extreme point solution provides a projected certificate that can be computed in polynomial time~\cite{SchrijverIntBook}; we just need to verify that there cannot exist a projected certificate with smaller support.  Let $\lambda^{*}$ be an extreme point optimal solution, and by contradiction assume that $\wt{\lambda}$ gives a projected certificate and is such that $\supp(\wt{\lambda})$ is strictly contained in $\supp(\lambda^*)$. Since $\wt{\lambda} \ge 0$ and also different from 0, by scaling we can assume without loss of generality that $e^T \wt{\lambda} = 1$, and thus $\wt{\lambda}$ is a feasible solution for the LP above. This implies that
\begin{eqnarray*}
 B^T \left(\lambda^{*} - \wt{\lambda} \right) &=& 0\\
e^T \left(\lambda^{*} - \wt{\lambda} \right) &=& 0,
\end{eqnarray*}
so the assumption $\supp(\wt{\lambda}) \subsetneq \supp(\lambda^*)$ implies that the columns of the matrix $\left [\begin{array}{c} B^T\\ e^T\end{array}\right]$ in the support of $\lambda^{*}$ are linearly dependent. But since $\lambda^*$ is an extreme point, it is a basic solution, namely the columns of the matrix in the support of $\lambda^*$ are linearly independent. This reaches a contradiction and concludes the proof.

\section{Original Feasibility Pump stalls even when flipping variables with zero fractionality is allowed}\label{app:3choice}

In Section \ref{sec:contrib} we showed that the original Feasibility Pump without restarts may stall; we now show that this is still the case even if variables with zero fractionality can be flipped in the perturbation step. 

	Let $TT$, the number of variables to be flipped, be randomly selected from the set $[t, T] \cap \mathbb{Z}$, where $T \in \mathbb{Z}_{++}$ is a pre-determined constant in the FP code (independent of the instance).
Moreover assume the reasonable convention that for two variables with equal fractionality, we break ties using their index number, that is, if the $x_i$ and $x_j$ have the same fractionality and $i < j$, then $x_i$ is picked before $x_j$ to be flipped.

	Consider the following subset-sum problem:
	\begin{eqnarray*}
	\textup{max}& x_{T + 2} \\
	\textup{s.t.}&	5x_1 + \dots+ 5x_{T + 1} + 2x_{T + 2}  = 5T + 5\\
	& x_i \in \{0, 1\} \ \forall \ i \in [T + 2]
	\end{eqnarray*}

	Clearly the LP optimal solution $\bar{x}^0$ is of the form $\bar{x}^0_{T + 2} = 1$, $\bar{x}^0_i = \frac{3}{5}$ for some $i \in [T + 1]$  and $\bar{x}^0_j = 1$ for all $j \in [T + 1] \setminus\{ i\}$. Rounding this we obtain $\wt{x}^0$ which is of the form $\wt{x}^0_{T +2} = 1$ and $\wt{x}^0_j = 1$ for all $j \in [T + 1]$. It is also straightforward to verify that $\wt{x}^0$ is a stalling solution (see Definition~\ref{defn:stall}). So that algorithm randomly selects $TT$ from the set $[t, T] \cap \mathbb{Z}$ and flips $TT$ variables. Note that only $x_i$ has a fractionality of $|\frac{3}{5} - 1|$ and all the other variables have a fractionality of $0$ for some $i \in [T + 1]$. So using the convention for breaking ties, we flip $x_i$ and $TT - 1$ other variables. Since $TT \leq T < T +1$, the new point $\wt{x}$ is of the form $\wt{x}_{T + 2} = 1$ and $\wt{x}_j = 0 $ for $j \in S \subseteq [T +1]$ and $ \wt{x}_j = 1$ for $j \in [T +1]\setminus S$. (Note that $S$ can also be $\emptyset$ since we make no assumption on $t$).

First note that $\wt{x}$ is not a feasible solution since $\wt{x}_{T + 2} = 1$. Moreover,
\begin{enumerate}
\item If $S = \emptyset$, then $\wt{x} = \wt{x}^0$, a stalling solution visited before.
\item If $S \neq \emptyset$, then $5\wt{x}_1 + \dots+ 5\wt{x}_{T + 1} + 2\wt{x}_{T + 2} < 5T + 5$ and on projecting to the LP relaxation we will obtain a point of the form of $\bar{x}^0$. Rounding this again gives us $\wt{x}^0$, a stalling solution visited before.
\end{enumerate}
This completes the proof.

\section{No long cycles in stalling}\label{sec:side}
\begin{lemma}\label{prop:convergent}
Suppose that following is a sequence of points visited by Feasibility Pump (without any randomization): 
$$ (\bar x^1, \bar y^1) \rightarrow (\tilde{x}^1, \bar y^1) \rightarrow  (\bar x^2, \bar y^2) \rightarrow (\tilde{x}^2, \bar y^2), $$
where $(\bar x^i, \bar y^i),$ $i \in \{1,2\}$ are the vertices of the LP relaxation, $\tilde{x}^i,$ $i \in \{1,2\}$ are $0-1$ vectors, $\tilde{x}^i = \round(\bar x^i)$ and $(\bar x^2, \bar y^2) = \lproj(\tilde{x}^1, \bar y^1) $. Then,
$$ \|\bar{x}^1 - \tilde{x}^1 \|_1 \geq \|\bar{x}^2 - \tilde{x}^2 \|_1.$$ 
\end{lemma}
\begin{proof}
This result holds due to the fact that we are sequentially projecting using the same norm. In particular, we have that
$$ \|\bar{x}^1 - \tilde{x}^1 \|_1 \geq \|\bar{x}^2 - \tilde{x}^1 \|_1,$$ 
since $(\bar x^2, \bar y^2) = \lproj(\tilde{x}^1, \bar y^1) $, i.e., $\bar{x}^2$ is a closest point in $l_1$-norm to $\tilde{x}^1$ in the projection of the LP relaxation in the $x$-space. Then 
$$ \|\bar{x}^2 - \tilde{x}^1 \|_1 \geq \|\bar{x}^2 - \tilde{x}^2 \|_1,$$ 
since $\tilde{x}^1$ and $\tilde{x}^2$ are both integer points and $\tilde{x}^2$ is obtained by rounding  $\bar{x}^2$ (and a rounded point is the closest integer point in $\ell_1$ norm).	
\end{proof}

A \emph{long cycle} in feasibility pump is a sequence 
$$ (\bar x^1, \bar y^1) \rightarrow (\tilde{x}^1, \bar y^1) \rightarrow  (\bar x^2, \bar y^2) \rightarrow (\tilde{x}^2, \bar y^2) \rightarrow \dots  (\bar x^k, \bar y^k) \rightarrow (\tilde{x}^k, \bar y^k)$$

where 
\begin{enumerate}
\item $(\bar x^i, \bar y^i),$ $i \in \{1,2, \dots, k\}$ are the vertices of the LP relaxation, $\tilde{x}^i,$ $i \in \{1,2, \dots, k\}$ are $0-1$ vectors, $\tilde{x}^i = \round(\bar x^i)$ and $(\bar x^{i+1}, \bar y^{i+1}) = \lproj(\tilde{x}^i, \bar y^{i})$,
\item $\tilde{x}^1, \tilde{x}^2, \dots, \tilde{x}^{k-1}$ are unique integer vectors, 
\item $\bar{x}^1  = \bar{x}^k$, $\tilde{x}^1 = \tilde{x}^k$, and
\item $k \geq 3$.
\end{enumerate}
The statement of Theorem \ref{thm:sidenote} is that such a scenario cannot occur, assuming $0.5$ is always rounded consistently. 
\begin{proof}[Proof of Theorem \ref{thm:sidenote}] Without loss of generally, we assume that $0.5$ is rounded up to $1$.  Consider the sub-sequence $(\bar x^i, \bar y^i) \rightarrow (\tilde{x}^i, \bar y^i) \rightarrow  (\bar x^{i+1}, \bar y^{i+1}) \rightarrow (\tilde{x}^{i+1}, \bar y^{i+1}) $. By Lemma \ref{prop:convergent}, since there is cycling, we have that 

$$\|\bar{x}^i - \tilde{x}^i \|_1 = \|\bar{x}^{i+1} - \tilde{x}^i \|_1 = \|\bar{x}^{i+1} - \tilde{x}^{i+1} \|_1. $$

For simplicity and without loss of generality, we may assume that $\tilde{x}^i$ is the all ones vector. (This can be achieved by reflecting on coordinates the LP relaxation and the $[0, \ 1]^n$ hypercube. Note that under such mappings, the sequence of points in feasibility pump will not be altered. Moreover, a point with value $0.5$ in some coordinates $I \subseteq [n]$ will be mapped to a point with $0.5$ in the coordinates $I$.)

Let $\emptyset \neq J \subseteq [n]$ be the set of indices where $\tilde{x}^i_j \neq \tilde{x}^{i+1}_j$, that is $\tilde{x}^{i +1}_j = 0$ for all $j \in J$. Since $\|\bar{x}^{i+1} - \tilde{x}^i \|_1 = \|\bar{x}^{i+1} - \tilde{x}^{i+1} \|_1$, we have 
\begin{eqnarray}
\sum_{j = 1}^n (1 - \bar{x}^{i+ 1}_j) &=& \sum_{j \in [n]\setminus J} (1 - \bar{x}^{i+ 1}_j)  + \sum_{j \in J} \bar{x}^{i+ 1}_j  \nonumber \\
\label{eq:sumofflip}\Leftrightarrow  \sum_{j \in J} \bar{x}^{i+ 1}_j &=& \frac{|J|}{2}.
\end{eqnarray}
Now observe that since $\tilde{x}^{i +1}_j = 0$ for $j \in J$, we must have that $\bar{x}^{i +1}_j < 0.5$ for all $j \in J$. This contradicts, (\ref{eq:sumofflip}).
 
\end{proof}
\end{document}